\documentclass{article}

\usepackage{amsmath,amssymb,amsfonts,amsthm,bm,mathrsfs,color,natbib}
\usepackage[hyphens]{url}
\usepackage[colorlinks,citecolor=blue,urlcolor=blue]{hyperref}

\title{The relationship between the transition functions of the labeled and unlabeled versions of the infinitely-many-neutral-alleles diffusion model}
\author{S. N. Ethier\thanks{\,Department of Mathematics, University of Utah. Email: \href{mailto:ethier@math.utah.edu}{ethier@math.utah.edu}.}}
\date{}

\theoremstyle{plain}
\newtheorem{theorem}{Theorem}
\newtheorem{lemma}[theorem]{Lemma}

\theoremstyle{remark}
\newtheorem*{remark}{Remark} % remarks unnumbered 

\allowdisplaybreaks

\begin{document}
\maketitle

\begin{abstract}
The transition function of the unlabeled infinitely-many-neutral-alleles diffusion model, as expressed by \cite{Z15}, is derived from the transition function of the labeled infinitely-many-neutral-alleles diffusion model, slightly simplifying the derivation by \cite{F10}.  
\end{abstract}

\section{Introduction}

Given $\theta>0$, let $\{D_t,\,t\ge0\}$ be the pure death process in $\textbf{Z}_+\cup\{\infty\}$ starting at $\infty$ with death rates
\begin{equation*}
\lambda_n:=\textstyle{\frac12}n(n-1+\theta), \quad n\ge0,
\end{equation*}
($\infty$ is an entrance boundary) and define
\begin{equation*}
d_n^\theta(t):=\textbf{P}(D_t=n),\quad n\ge0,\,t>0.  
\end{equation*}
Explicit but complicated formulas are available for these functions (\citeauthor{T84}, \citeyear{T84}); they will not be needed here.

The \textit{unlabeled} infinitely-many-neutral-alleles diffusion model was characterized by \cite{EK81} as a diffusion process in the Kingman simplex
\begin{equation*}
\overline{\nabla}_\infty:=\bigg\{\bm p=(p_1,p_2,\ldots)\in[0,1]^\infty:p_1\ge p_2\ge\cdots\ge0,\,\sum_{i=1}^\infty p_i\le1\bigg\}.
\end{equation*}
with a specified generator depending on $\theta$.  Its transition function was shown by \cite{Z15} to have the form
\begin{equation}\label{Q}
Q(t,\bm p,\cdot)=\sum_{n=0}^\infty d_n^\theta(t)\,\Lambda_n^\theta(\bm p,\cdot),
\end{equation}
where $\Lambda_n^\theta(\bm p,\cdot)$ is a certain one-step transition function on $\overline{\nabla}_\infty$ for each $n\ge0$.

The \textit{labeled} infinitely-many-neutral-alleles diffusion model is a Fleming--Viot process in $\mathscr{P}(S)$, the space of Borel probability measures on a compact metric space $S$, with mutation operator $A$ on $C(S)$ given by
\begin{equation*}
(Af)(x):={\textstyle{\frac12}}\theta\int_S(f(\xi)-f(x))\,\nu_0(d\xi),
\end{equation*}
where $\nu_0\in\mathscr{P}(S)$ is assumed to be nonatomic.  Its transition function was shown by \cite{EG93} to have the form
\begin{equation}\label{P}
P(t,\mu,\cdot)=\sum_{n=0}^\infty d_n^\theta(t)\,\Gamma_n^{\theta,\nu_0}(\mu,\cdot),
\end{equation}
where $\Gamma_n^{\theta,\nu_0}(\mu,\cdot)$ is a certain one-step transition function on $\mathscr{P}(S)$ for each $n\ge0$.

Let the mapping $\Phi:\mathscr{P}(S)\mapsto\overline{\nabla}_\infty$ be given by
\begin{equation*}
\Phi(\mu):=(p_{(1)},p_{(2)},\ldots),
\end{equation*}
where $p_{(1)},p_{(2)},\ldots$ are the descending order statistics of the sizes (or masses) of the atoms of $\mu$.  We would like to apply the Markov mapping theorem of \cite{D65} (see the next section) to conclude that, for each $t>0$ and $\bm p\in\overline{\nabla}_\infty$,
\begin{equation}\label{Dynkin}
Q(t,\bm p,\cdot)=P(t,\mu,\Phi^{-1}(\cdot))\text{ whenever }\mu\in\Phi^{-1}(\bm p), 
\end{equation}
thereby deriving $Q(t,\bm p,\cdot)$ from $P(t,\mu,\cdot)$.  In fact, this has already been done by \cite{F10}, but we want to give a more direct derivation by showing that, for each $n\ge0$ and $\bm p\in\overline{\nabla}_\infty$,
\begin{equation*}
\Lambda_n^\theta(\bm p,\cdot)=\Gamma_n^{\theta,\nu_0}(\mu,\Phi^{-1}(\cdot))\text{ whenever }\mu\in\Phi^{-1}(\bm p).
\end{equation*}
\citeauthor{F10}'s (\citeyear{F10}) derivation was based on \citeauthor{E92}'s (\citeyear{E92}) eigenfunction expansion for $Q(t,\bm p,\cdot)$ instead of \citeauthor{Z15}'s (\citeyear{Z15}) more elegant formula \eqref{Q} (and \eqref{Lambda} below).  In particular, Feng's derivation relied on the explicit formulas for the functions $d_n^\theta(t)$ and ours does not.  \citeauthor{Z15}'s (\citeyear{Z15}) proof of \eqref{Q} also relied on these formulas, but \citeauthor{EG93}' (\citeyear{EG93}) proof of \eqref{P} (and \eqref{Gamma} below) used instead the fact that the functions $d_n^\theta(t)$ satisfy the Kolmogorov forward equations.

There is an alternative version of the Markov mapping theorem due to \cite{RP81} (see the next section), which can also be used to prove \eqref{Dynkin}.

Let us specify $\Lambda_n^\theta(\bm p,\cdot)$ and $\Gamma_n^{\theta,\nu_0}(\mu,\cdot)$.

First, Kingman's Poisson--Dirichlet distribution with parameter $\theta>0$, denoted by $\text{PD}_\theta$, is the unique stationary distribution for the transition function $Q(t,\bm p,\cdot)$, and it is concentrated on
\begin{equation*}
\nabla_\infty:=\bigg\{\bm p=(p_1,p_2,\ldots)\in[0,1]^\infty:p_1\ge p_2\ge\cdots\ge0,\,\sum_{i=1}^\infty p_i=1\bigg\},
\end{equation*}  
a dense subset of $\overline{\nabla}_\infty$ as the notation suggests.

Define the functions $P_{\bm n}\in C(\overline{\nabla}_\infty)$ for $\bm n=(n_1,\ldots,n_l)$ with $n_1\ge\cdots\ge n_l\ge1$ and $l\in\textbf{N}$ by
\begin{equation}\label{sampling}
P_{\bm n}(\bm p):=\binom{n}{n_1,\ldots,n_l}\frac{1}{\alpha_1!\cdots \alpha_n!}\sum_{i_1,\ldots,i_l\in\textbf{N}\text{ distinct}} p_{i_1}^{n_1}\cdots p_{i_l}^{n_l},\quad \bm p\in\nabla_\infty,
\end{equation}
where $n=|\bm n|:=n_1+\cdots+n_l$ and $\alpha_j=|\{i:n_i=j\}|$ for $j=1,\ldots,n$. These functions, defined on $\nabla_\infty$, extend uniquely to $C(\overline{\nabla}_\infty)$.  Observe that \eqref{sampling} is the probability that a random sample of size $n$ from a population with allele frequencies $p_1,p_2,\ldots$ contains $l$ alleles, one with $n_1$ representatives, another with $n_2$ representatives, and so on.

We can now define 
\begin{equation}\label{Lambda}
\Lambda_n^\theta(\bm p,d\bm q):=\begin{cases}{\rm PD}_\theta(d\bm q)&\text{if $n=0$ or 1}\\ \noalign{\medskip}
\displaystyle{\sum_{\bm n:|\bm n|=n}\frac{P_{\bm n}(\bm p)P_{\bm n}(\bm q)}{\int_{\overline{\nabla}_\infty}P_{\bm n}\,d{\rm PD}_\theta}\,{\rm PD}_\theta(d\bm q)}&\text{if $n\ge2$}\end{cases}
\end{equation}
for $\bm p\in \overline{\nabla}_\infty$.  The integral in the denominator is given by the Ewens sampling formula, namely
\begin{equation}\label{ESF}
\int_{\overline{\nabla}_\infty}P_{\bm n}\,d\text{PD}_\theta=\frac{n!}{n_1\cdots n_l}\,\frac{1}{\alpha_1!\cdots\alpha_n!}\,\frac{\theta^l}{\theta_{(n)}},
\end{equation}
where $a_{(n)}:=a(a+1)\cdots(a+n-1)$.

For $\beta>0$ and $\nu\in\mathscr{P}(S)$, define $\Pi_{\beta,\nu}\in\mathscr{P}(\mathscr{P}(S))$ by
\begin{equation*}
\Pi_{\beta,\nu}(\cdot):=\textbf{P}\bigg(\sum_{i=1}^\infty \rho_i\,\delta_{\xi_i}\in\cdot\bigg),
\end{equation*}
where $(\rho_1,\rho_2,\ldots)$ is $\text{PD}_\beta$ and independent of $\xi_1,\xi_2,\ldots$, which are i.i.d.\ $\nu$.  Here $\delta_x\in\mathscr{P}(S)$ is the unit mass at $x\in S$.  The unique stationary distribution of $P(t,\mu,\cdot)$ is $\Pi_{\theta,\nu_0}$, and we can now define
\begin{equation}\label{Gamma}
\Gamma_n^{\theta,\nu_0}(\mu,\cdot):=\begin{cases}\Pi_{\theta,\nu_0}(\cdot)&\text{if $n=0$}\\ \noalign{\medskip}
\displaystyle{\int_{S^n}\mu^n(dx_1\times\cdots\times dx_n)}&\\ \noalign{\vglue-1mm}
\qquad\quad{}\cdot\Pi_{n+\theta,(n+\theta)^{-1}(\delta_{x_1}+\cdots+\delta_{x_n}+\theta\nu_0)}(\cdot)&\text{if $n\ge1$}\end{cases}
\end{equation}
for $\mu\in\mathscr{P}(S)$, where $\mu^n\in\mathscr{P}(S^n)$ is the $n$-fold product measure $\mu\times\cdots\times\mu$.

Let us state formally the result we aim to establish.

\begin{theorem}\label{one}
For each $n\ge0$ and $\bm p\in\overline{\nabla}_\infty$,
\begin{equation}\label{Dynkin-n}
\Lambda_n^\theta(\bm p,\cdot)=\Gamma_n^{\theta,\nu_0}(\mu,\Phi^{-1}(\cdot))\text{ whenever }\mu\in\Phi^{-1}(\bm p).
\end{equation}
\end{theorem}

This immediately implies that, for each $t>0$ and $\bm p\in\overline{\nabla}_\infty$, \eqref{Dynkin} holds.

\section{The Markov mapping theorem}

There are two versions of the Markov mapping theorem that can be expressed in terms of transition functions, one due to \cite{D65} and another due to \cite{RP81}.

\begin{theorem}[\citeauthor{D65}, \citeyear{D65}, \S6.10]\label{Dynkin-Th}
Let $E$ and $E_0$ be complete separable metric spaces, and let $\Phi:E\mapsto E_0$ be Borel measurable and surjective.  Let $P(t,x,\cdot)$ be the transition function for a Markov process $X$ in $E$. If for each $t>0$ the condition 
\begin{equation}\label{Dynkin-cond}
P(t,x,\Phi^{-1}(\cdot))\text{ depends on }x\text{ only through }\Phi(x)
\end{equation}
holds, then $Y:=\Phi\circ X$ is a Markov process in $E_0$ with transition function
\begin{equation}\label{Q-def}
Q(t,y,\cdot):=P(t,x,\Phi^{-1}(\cdot)),\text{ where }x\in\Phi^{-1}(y).
\end{equation}
\end{theorem}

\begin{theorem}[\citeauthor{RP81}, \citeyear{RP81}]\label{RP-Th}
Let $E$ and $E_0$ be complete separable metric spaces, and let $\Phi:E\mapsto E_0$ be Borel measurable and surjective.  Let $\lambda$ be a one-step transition function from $E_0$ to $E$ $($i.e., $y\in E_0\mapsto\lambda(y,\cdot)\in\mathscr{P}(E)$ is Borel measurable$)$ satisfying $\lambda(y,\Phi^{-1}(y))=1$ for all $y\in E_0$.  Let $P(t,x,\cdot)$ be the transition function for a Markov process $X$ in $E$ with initial distribution $\mu:=\int_{E_0}\lambda(y,\cdot)\,\mu_0(dy)$ for some $\mu_0\in\mathscr{P}(E_0)$.  If the condition
\begin{equation}\label{RP-cond}
\int_E P(t,x,\cdot)\,\lambda(y,dx)=\int_E\bigg[\int_{E_0}\lambda(z,\cdot)\,P(t,x,\Phi^{-1}(dz))\bigg]\lambda(y,dx),
\end{equation}
holds for every $t>0$ and $y\in E_0$, then $Y:=\Phi\circ X$ is a Markov process in $E_0$ with initial distribution $\mu_0$ and transition function
\begin{equation}\label{Q-def-RP}
Q(t,y,\cdot):=\int_E P(t,x,\Phi^{-1}(\cdot))\,\lambda(y,dx).
\end{equation}
\end{theorem}

\begin{remark}
Under the conditions of Theorem~\ref{RP-Th}, one can show that
\begin{equation*}
\textbf{P}_{\mu}(X(t)\in A\mid \Phi\circ X(s),\,0\le s\le t)=\lambda(\Phi\circ X(t),A)
\end{equation*}
for all $A\in\mathscr{B}(E)$, which gives an interpretation to $\lambda$.
\end{remark}

Theorem~\ref{RP-Th} is intended for situations in which Theorem~\ref{Dynkin-Th} does not apply.  If Theorem~\ref{Dynkin-Th} does apply, then 
Theorem~\ref{RP-Th} applies as well provided one additional condition, namely \eqref{RP-cond}, is satisfied, which in this case reduces to
\begin{equation}\label{RP-cond-2}
\int_E P(t,x,\cdot)\,\lambda(y,dx)=\int_{E_0}\lambda(z,\cdot)\,P(t,w,\Phi^{-1}(dz))
\end{equation}
for every $t>0$ and $y\in E_0$, where $w\in\Phi^{-1}(y)$.  Furthermore, Theorem~\ref{RP-Th} yields a slightly weaker conclusion because $X$ is required to have an initial distribution of a particular form that depends on $\lambda$.

Nevertheless, we will see that both theorems are applicable to our problem.

%If one is interested primarily in showing that $Y:=\Phi\circ X$ is a Markov process in $E_0$, one can use versions of these theorems expressed in terms of generators, such as \citeauthor{K98} (\citeyear{K98}, Corollary 3.5).

\section{Proof of Theorem~\ref{one}}

The $n=0$ case of \eqref{Dynkin-n} is clear: For each $\bm p\in\overline{\nabla}_\infty$,
\begin{equation*}
\Lambda_0^\theta(\bm p,\cdot)=\text{PD}_\theta(\cdot)=\textbf{P}\bigg(\Phi\bigg(\sum_{i=1}^\infty \rho_i\,\delta_{\xi_i}\bigg)\in\cdot\bigg)=\Pi_{\theta,\nu_0}(\Phi^{-1}(\cdot))=\Gamma_0^{\theta,\nu_0}(\mu,\Phi^{-1}(\cdot)),
\end{equation*}
regardless of $\mu\in\mathscr{P}(S)$.

We turn to $n\ge1$. A key step in the argument is given in the proof of Corollary~1.3 of \cite{EG93}.  There it is shown that
\begin{align}\label{rv-rep}
&\int_{S^n}\mu^n(dx_1\times\cdots\times dx_n)\,\Pi_{n+\theta,(n+\theta)^{-1}(\delta_{x_1}+\cdots+\delta_{x_n}+\theta\nu_0)}(\Phi^{-1}(\cdot))\nonumber\\
&\quad{}=\int_{S^n}\mu^n(dx_1\times\cdots\times dx_n)\,\textbf{P}\bigg(\Phi\bigg(\varepsilon\sum_{i=1}^n U_i^{(n)}\delta_{x_i}+(1-\varepsilon)\sum_{i=1}^\infty\rho_i\delta_{\xi_i}\bigg)\in\cdot\bigg),
\end{align}
where $\varepsilon$ is beta$(n,\theta)$, $(U_1^{(n)},\ldots,U_n^{(n)})$ is uniformly distributed over the $n$-simplex
\begin{equation*}
\Delta_n:=\bigg\{(p_1,\ldots,p_n)\in[0,1]^n:\sum_{i=1}^n p_i=1\bigg\}
\end{equation*}
(or, $(U_1^{(n)},\ldots,U_n^{(n)})$ is Dirichlet$(1,1,\ldots,1)$),
$(\rho_1,\rho_2,\ldots)$ is $\text{PD}_\theta$, $\xi_1,\xi_2,\ldots$ are i.i.d.\ $\nu_0$, and $\varepsilon$, $(U_1^{(n)},\ldots,U_n^{(n)})$, $(\rho_1,\rho_2,\ldots)$, and $(\xi_1,\xi_2,\ldots)$ are independent.

We let
\begin{equation*}
\overline{\Delta}_\infty:=\bigg\{\bm p=(p_1,p_2,\ldots)\in[0,1]^\infty:\sum_{i=1}^\infty p_i\le1\bigg\}
\end{equation*}
and define the function $\rho:\overline{\Delta}_\infty\mapsto\overline{\nabla}_\infty$ by
\begin{equation*}
\rho(p_1,p_2,\ldots):=(p_{(1)},p_{(2)},\ldots),
\end{equation*}
where $p_{(1)},p_{(2)},\ldots$ are the descending order statistics of $p_1,p_2,\ldots$.  And, given $\bm p\in\overline{\nabla}_\infty$, we let
\begin{equation}\label{p_0}
p_0:=1-\sum_{i=1}^\infty p_i,
\end{equation}
so that $\bm p\in\nabla_\infty$ if and only if $p_0=0$.

The case $n=1$ of \eqref{Dynkin-n} has already been established in the proof of the corollary cited above.  There it was shown that, for every $\mu\in\mathscr{P}(S)$,
\begin{align*}
&\int_S \mu(dx)\,\Pi_{1+\theta,(1+\theta)^{-1}(\delta_x+\theta\nu_0)}(\Phi^{-1}(\cdot))\\
&\quad{}=\int_S \mu(dx)\,\textbf{P}(\rho(\varepsilon,(1-\varepsilon)\rho_1,(1-\varepsilon)\rho_2,\ldots)\in\cdot)\\
&\quad{}=\text{PD}_\theta(\cdot),
\end{align*}
where $\varepsilon$ is beta$(1,\theta)$, while $(\rho_1,\rho_2,\ldots)$ is $\text{PD}_\theta$ and independent of $\varepsilon$. 
The first equality uses \eqref{rv-rep} and the second relies on the GEM representation of $\text{PD}_\theta$.  It follows that, for each $\bm p\in\overline{\nabla}_\infty$,
\begin{equation*}
\Lambda_1^\theta(\bm p,\cdot)=\text{PD}_\theta(\cdot)=\int_S \mu(dx)\,\Pi_{1+\theta,(1+\theta)^{-1}(\delta_x+\theta\nu_0)}(\Phi^{-1}(\cdot))=\Gamma_1^{\theta,\nu_0}(\mu,\Phi^{-1}(\cdot)),
\end{equation*}
regardless of $\mu\in\mathscr{P}(S)$.

We now turn to the general case, $n\ge2$.  We need to show that, for each $\bm p\in\overline{\nabla}_\infty$,
\begin{align*}
&\sum_{\bm n:|\bm n|=n}\frac{P_{\bm n}(\bm p)P_{\bm n}(\bm q)}{\int_{\overline{\nabla}_\infty}P_{\bm n}\,d{\rm PD}_\theta}\,{\rm PD}_\theta(d\bm q)\\
&\quad{}=\int_{S^n}\mu^n(dx_1\times\cdots\times dx_n)\,\Pi_{n+\theta,(n+\theta)^{-1}(\delta_{x_1}+\cdots+\delta_{x_n}+\theta\nu_0)}(\Phi^{-1}(d\bm q))
\end{align*}
whenever $\mu\in\Phi^{-1}(\bm p)$.  Given $\gamma,\lambda\in\mathscr{P}(\overline{\nabla}_\infty)$,
to show that $\gamma=\lambda$, it is enough to show that, for each $g$ in a separating class of functions in $C(\overline{\nabla}_\infty)$, $\int g\,d\gamma=\int g\,d\lambda$.  The separating class we use is the set of all $P_{\bm m}\in C(\overline{\nabla}_\infty)$ for $\bm m=(m_1,\ldots,m_k)$ with $m_1\ge\cdots\ge m_k\ge1$ and $k\in\textbf{N}$.  Lemma 3.1 of \cite{E92} confirms that this class is separating (because its linear span is dense in $C(\overline{\nabla}_\infty)$). So we fix an arbitrary such $\bm m$.  It remains to show that, for each $\bm p\in\overline{\nabla}_\infty$,
\begin{align}\label{identity}
&\int_{\overline{\nabla}_\infty}P_{\bm m}(\bm q)\sum_{\bm n:|\bm n|=n}\frac{P_{\bm n}(\bm p)P_{\bm n}(\bm q)}{\int_{\overline{\nabla}_\infty}P_{\bm n}\,d{\rm PD}_\theta}\,{\rm PD}_\theta(d\bm q)\nonumber\\
&{}=\int_{\overline{\nabla}_\infty}P_{\bm m}(\bm q)\int_{S^n}\mu^n(dx_1\times\cdots\times dx_n)\,\Pi_{n+\theta,(n+\theta)^{-1}(\delta_{x_1}+\cdots+\delta_{x_n}+\theta\nu_0)}(\Phi^{-1}(d\bm q))
\end{align}
whenever $\mu\in\Phi^{-1}(\bm p)$.  We show this by evaluating each side separately.

The left side of \eqref{identity} is, for given $\bm p\in\overline{\nabla}_\infty$,
\begin{equation}\label{LHS}
\sum_{\bm n:|\bm n|=n}\frac{\int_{\overline{\nabla}_\infty}P_{\bm m}P_{\bm n}\,d\text{PD}_\theta}{\int_{\overline{\nabla}_\infty}P_{\bm n}\,d\text{PD}_\theta}\,P_{\bm n}(\bm p),
\end{equation}
and the right side is, for given $\bm p\in\overline{\nabla}_\infty$,
\begin{align}\label{RHS}
&\int_{\overline{\nabla}_\infty}P_{\bm m}(\bm q)\int_{S^n}\mu^n(dx_1\times\cdots\times dx_n)\nonumber\\ \noalign{\vglue-2mm}
&\qquad\qquad\qquad\qquad{}\cdot\textbf{P}\bigg(\Phi\bigg(\varepsilon\sum_{i=1}^n U_i^{(n)}\delta_{x_i}+(1-\varepsilon)\sum_{i=1}^\infty\rho_i\delta_{\xi_i}\bigg)\in d\bm q\bigg),
\end{align}
where $\mu\in\Phi^{-1}(\bm p)$.  Here we have used \eqref{rv-rep}.  

Eq.\ \eqref{RHS}, in which $x_1,\ldots,x_n$ can be viewed as a random sample from $\mu$, can be better understood by first considering the special case $n=2$.  In that case, this is
\begin{align}\label{RHS2}
&\int_{\overline{\nabla}_\infty}P_{\bm m}(\bm q)\int_{S^2}\mu^2(dx_1\times dx_2)\nonumber\\ \noalign{\vglue-2mm}
&\qquad\qquad\qquad\qquad{}\cdot\textbf{P}\bigg(\Phi\bigg(\varepsilon[U\delta_{x_1}+(1-U)\delta_{x_2}]+(1-\varepsilon)\sum_{i=1}^\infty\rho_i\delta_{\xi_i}\bigg)\in d\bm q\bigg),
\end{align}
where $\varepsilon$ is beta$(2,\theta)$, $U$ is uniform$(0,1)$, $(\rho_1,\rho_2,\ldots)$ is $\text{PD}_\theta$, and $\xi_1,\xi_2,\ldots$ are i.i.d.\ $\nu_0$, and all are independent.  If $(x_1,x_2)$ is a random sample from $\mu$ and $\Phi(\mu)=\bm p\in\overline{\nabla}_\infty$, the probability that $x_1=x_2$ is
\begin{equation*}
\sum_{i=1}^\infty p_i^2=P_{(2)}(\bm p)
\end{equation*}
and the probability that $x_1\ne x_2$ is the complementary probability, namely
\begin{align*}
1-\sum_{i=1}^\infty p_i^2&=\bigg(\sum_{i=0}^\infty p_i\bigg)^2-\sum_{i=0}^\infty p_i^2+p_0^2\\
&=\sum_{i,j\ge0:i\ne j}p_i p_j+p_0^2\\
&=\sum_{i,j\ge1:i\ne j} p_i p_j+2p_0(1-p_0)+p_0^2\\
&=P_{(1,1)}(\bm p).
\end{align*}
The result is that \eqref{RHS2} equals
\begin{align*}
&\int_{\overline{\nabla}_\infty}P_{\bm m}(\bm q)\,\textbf{P}(\rho(\varepsilon,(1-\varepsilon)\rho_1,(1-\varepsilon)\rho_2,\ldots)\in d\bm q)\,P_{(2)}(\bm p)\\
&\qquad\quad{}+\int_{\overline{\nabla}_\infty}P_{\bm m}(\bm q)\,\textbf{P}(\rho(\varepsilon U,\varepsilon (1-U),(1-\varepsilon)\rho_1,(1-\varepsilon)\rho_2,\ldots)\in d\bm q)\,P_{(1,1)}(\bm p)\\
&\quad{}=\textbf{E}[P_{\bm m}(\rho(\varepsilon,(1-\varepsilon)\rho_1,(1-\varepsilon)\rho_2,\ldots))]\,P_{(2)}(\bm p)\\
&\qquad\quad{}+\textbf{E}[P_{\bm m}(\rho(\varepsilon U,\varepsilon (1-U),(1-\varepsilon)\rho_1,(1-\varepsilon)\rho_2,\ldots))]\,P_{(1,1)}(\bm p).
\end{align*}

More generally, \eqref{RHS} becomes
\begin{align}\label{RHS-1}
&\sum_{\bm n:|\bm n|=n}\int_{\overline{\nabla}_\infty}P_{\bm m}(\bm q)\textbf{P}(\rho(\varepsilon V_1^{\bm n},\ldots,\varepsilon V_l^{\bm n},(1-\varepsilon)\rho_1,(1-\varepsilon)\rho_2,\ldots)\in d\bm q)\,P_{\bm n}(\bm p)\nonumber\\
&\quad{}=\sum_{\bm n:|\bm n|=n}\textbf{E}[P_{\bm m}(\rho(\varepsilon V_1^{\bm n},\ldots,\varepsilon V_l^{\bm n},
(1-\varepsilon)\rho_1,(1-\varepsilon)\rho_2,\ldots))]\,P_{\bm n}(\bm p),
\end{align}
where $\varepsilon$ is beta$(n,\theta)$, $(V_1^{\bm n},\ldots,V_l^{\bm n})$ is Dirichlet$(\bm n)$, $(\rho_1,\rho_2,\ldots)$ and $\xi_1,\xi_2,\ldots$ are as before, and all are independent.

Comparing \eqref{LHS} and \eqref{RHS-1}, it therefore remains to show that
\begin{equation}\label{m-n-equation}
\frac{\int_{\overline{\nabla}_\infty}P_{\bm m}P_{\bm n}\,d\text{PD}_\theta}{\int_{\overline{\nabla}_\infty}P_{\bm n}\,d\text{PD}_\theta}
=\textbf{E}[P_{\bm m}(\rho(\varepsilon V_1^{\bm n},\ldots,\varepsilon V_l^{\bm n},
(1-\varepsilon)\rho_1,(1-\varepsilon)\rho_2,\ldots))]
\end{equation}
for $\bm m=(m_1,\ldots,m_k)$ and $\bm n=(n_1,\ldots,n_l)$ fixed but arbitrary.  We require $m_1\ge\cdots\ge m_k\ge1$, $n_1\ge\cdots\ge n_l\ge1$, and $k,l\in{\bf N}$, and we let $m=|\bm m|$ and $n=|\bm n|$.

Let us denote the sum in \eqref{sampling} by $P_{\bm n}^0(\bm p)$ (uniquely extended from $\nabla_\infty$ to $\overline{\nabla}_\infty$ by continuity).  Then, by \eqref{ESF},
\begin{equation}\label{ESF0}
\int_{\overline{\nabla}_\infty}P_{\bm n}^0\,d\text{PD}_\theta=(n_1-1)!\cdots(n_l-1)!\,\frac{\theta^l}{\theta_{(n)}}
\end{equation}
and $P_{\bm n}=C(\bm n)P_{\bm n}^0$ with
\begin{equation*}
C(\bm n):=\binom{n}{n_1,\ldots,n_l}\frac{1}{\alpha_1!\cdots\alpha_n!}.
\end{equation*}
Although \eqref{ESF} implicitly assumes $\bm n=(n_1,\ldots,n_l)$, $n_1\ge n_2\ge\cdots\ge n_l\ge1$, $l\in\textbf{N}$, and $n=|\bm n|$ (so that the probabilities sum to 1 for fixed $n$), it will be convenient to allow $n_1,\ldots,n_l$ to be arbitrary positive integers, not necessarily arranged in descending order.  Observe that $P_{\bm n}^0(\bm p)$ is the probability that in an ordered random sample of size $n$ from a population with allele frequencies $p_1,p_2,\ldots$, the first $n_1$ are of one allele, the next $n_2$ are of another allele, and so on.

The left side of \eqref{m-n-equation} is
\begin{equation}\label{LHS-1}
\frac{C(\bm m)\int_{\overline{\nabla}_\infty}P_{\bm m}^0 P_{\bm n}^0\,d\text{PD}_\theta}{\int_{\overline{\nabla}_\infty}P_{\bm n}^0\,d\text{PD}_\theta}.
\end{equation}
To see how this might be evaluated, we first consider the example $\bm m=(1,1)$ and $\bm n=(2,1)$.  Then \eqref{LHS-1} becomes
\begin{align}\label{1,1,2,1-example}
&C(1,1)\int_{\overline{\nabla}_\infty}P_{(1,1)}^0 P_{(2,1)}^0\,d\text{PD}_\theta\bigg/\int_{\overline{\nabla}_\infty}P_{(2,1)}^0\,d\text{PD}_\theta\nonumber\\
&\quad{}=\int_{\overline{\nabla}_\infty}(2P_{(3,2)}^0+2P_{(3,1,1)}^0+2P_{(2,2,1)}^0+P_{(2,1,1,1)}^0)\,d\text{PD}_\theta\bigg/\int_{\overline{\nabla}_\infty}P_{(2,1)}^0\,d\text{PD}_\theta\nonumber\\
&\quad{}=\bigg(2\cdot2\,\frac{\theta^2}{\theta_{(5)}}+2\cdot2\,\frac{\theta^3}{\theta_{(5)}}
+2\,\frac{\theta^3}{\theta_{(5)}}+\frac{\theta^4}{\theta_{(5)}}\bigg)\bigg/\frac{\theta^2}{\theta_{(3)}}\nonumber\\
&\quad{}=1-\frac{8+\theta}{(3+\theta)(4+\theta)}.
\end{align}
To explain the first equality, the integral in the numerator on the left side is the probability that an ordered  random sample of size 5 from $\text{PD}_\theta$ has the form $(x_1,x_2,x_3,x_3,x_4)$ with $x_1\ne x_2$ and $x_3\ne x_4$.  Let $D$ denote the event just mentioned.  Then
\begin{align*}
\textbf{P}(D)&=\textbf{P}(D\cap\{x_1=x_3,\,x_2=x_4\})+\textbf{P}(D\cap\{x_1=x_4,\,x_2=x_3\})\\
&\quad{}+\textbf{P}(D\cap\{x_1=x_3,\,x_2\ne x_4\})+\textbf{P}(D\cap\{x_2=x_3,\,x_1\ne x_4\})\\
&\quad{}+\textbf{P}(D\cap\{x_1=x_4,\,x_2\ne x_3\})+\textbf{P}(D\cap\{x_2=x_4,\,x_1\ne x_3\})\\
&\quad{}+\textbf{P}(D\cap\{x_1\ne x_3,\,x_1\ne x_4,\,x_2\ne x_3,\,x_2\ne x_4\}),
\end{align*}
which is the integral in the numerator on the right side of the first equality.  The second equality uses \eqref{ESF0}.

We can now state the general result.

\begin{lemma}[\citeauthor{F10}, \citeyear{F10}, Eq.~(5.101)]\label{two}
For $\bm m=(m_1,\ldots,m_k)$ with $m_1\ge\cdots\ge m_k\ge1$, $\bm n=(n_1,\ldots,n_l)$ with $n_1\ge\cdots\ge n_l\ge1$, $k,l\in{\bf N}$, $m=|\bm m|$, and $n=|\bm n|$,
\begin{align}\label{LHS-2}
\frac{\int_{\overline{\nabla}_\infty}P_{\bm m}P_{\bm n}\,d\text{PD}_\theta}{\int_{\overline{\nabla}_\infty}P_{\bm n}\,d\text{PD}_\theta}
&=C(\bm m)\sum_{r=0}^{k\wedge l}\bigg[\sum_{\substack{I\subset\{1,\ldots,k\}:\\ |I|=r}}\;\sum_{\substack{\gamma:I\mapsto\{1,\ldots,l\}\\ \text{\rm one-to-one}}}\nonumber\\
&\quad\;{}\dfrac{\prod_{i\in I}(m_i+n_{\gamma(i)}-1)!\,\prod_{i\in I^c}(m_i-1)!}{\prod_{i\in I}(n_{\gamma(i)}-1)!}\bigg]\frac{\theta^{k-r}}{(n+\theta)_{(m)}},
\end{align}
where $I^c:=\{1,\ldots,k\}-I$.
\end{lemma}

\begin{proof}
We calculate
\begin{equation*}
\int_{\overline{\nabla}_\infty}P_{\bm m}^0 P_{\bm n}^0\,d\text{PD}_\theta
=\int_{\overline{\nabla}_\infty}\bigg[\sum_{r=0}^{k\wedge l}\;\sum_{\substack{I\subset\{1,\ldots,k\}:\\ |I|=r}}\;\sum_{\substack{\gamma:I\mapsto\{1,\ldots,l\}\\ \text{is one-to-one}}}P_{\zeta^\circ(\bm m,\bm n,I,\gamma)}^0\bigg]\,d\text{PD}_\theta,
\end{equation*}
where first $\zeta(\bm m,\bm n,I,\gamma)\in(\textbf{Z}_+)^{k+l}$ is defined by
\begin{equation*}
\zeta(\bm m,\bm n,I,\gamma)_i:=\begin{cases}
m_i+n_{\gamma(i)}&\text{if $i\in I$},\\
m_i&\text{if $i\in\{1,\ldots,k\}-I$},\\
n_{i-k}&\text{if $i\in\{k+1,\ldots,k+l\}-\text{Range}(k+\gamma)$},\\
0&\text{if $i\in\text{Range}(k+\gamma)$}.
\end{cases}
\end{equation*}
and then $\zeta^\circ(\bm m,\bm n,I,\gamma)\in\textbf{N}^{k+l-r}$ deletes the $|I|$ zeros in $\zeta(\bm m,\bm n,I,\gamma)$.
It follows that
\begin{align*}
&\int_{\overline{\nabla}_\infty}P_{\bm m}^0 P_{\bm n}^0\,d\text{PD}_\theta\\
&\quad{}=\sum_{r=0}^{k\wedge l}\bigg[\sum_{\substack{I\subset\{1,\ldots,k\}:\\ |I|=r}}\;\sum_{\substack{\gamma:I\mapsto\{1,\ldots,l\}\\ \text{one-to-one}}}\;\prod_{i\in I}(m_i+n_{\gamma(i)}-1)!\prod_{i\in I^c}(m_i-1)!\\
&\qquad\qquad\qquad\qquad\qquad\qquad\qquad{}\cdot\prod_{i\in\text{Range}(\gamma)^c}(n_i-1)!\,\bigg]\frac{\theta^{k+l-r}}{\theta_{(m+n)}},
\end{align*}
where $I^c:=\{1,\ldots,k\}-I$ and $\text{Range}(\gamma)^c:=\{1,\ldots,l\}-\text{Range}(\gamma)$,
and therefore \eqref{LHS-2} follows.
\end{proof}

We turn to the evaluation of the right side of \eqref{m-n-equation}.  Just as we did for the left side, we first consider the example $\bm m=(1,1)$ and $\bm n=(2,1)$.  The simplest approach is to notice that $P_{(1,1)}=1-P_{(2)}$, but we prefer to use a method that generalizes to arbitrary $\bm m$ and $\bm n$.  Thus,
\begin{align*}
&\textbf{E}[P_{(1,1)}(\rho(\varepsilon V_1^{(2,1)},\varepsilon V_2^{(2,1)},(1-\varepsilon)\rho_1,(1-\varepsilon)\rho_2,\ldots))]\\
&\quad{}=\int_{\overline{\nabla}_\infty}P_{(1,1)}(\bm p)\,\textbf{P}(\rho(\varepsilon V_1^{(2,1)},\varepsilon V_2^{(2,1)},(1-\varepsilon)\rho_1,(1-\varepsilon)\rho_2,\ldots)\in d\bm p)\\
&\quad{}=\int_{\overline{\nabla}_\infty}\bigg(2p_1 p_2+2p_1\sum_{j\ge3} p_j+2p_2\sum_{j\ge3} p_j+\sum_{i,j\ge3:i\ne j} p_i p_j\bigg)\\
&\qquad\qquad\qquad{}\cdot\textbf{P}(\rho(\varepsilon V_1^{(2,1)},\varepsilon V_2^{(2,1)},(1-\varepsilon)\rho_1,(1-\varepsilon)\rho_2,\ldots)\in d\bm p)\\
&\quad{}=2\textbf{E}[\varepsilon^2]\,\textbf{E}[V_1^{(2,1)}V_2^{(2,1)}]+2\textbf{E}[\varepsilon(1-\varepsilon)]\,\textbf{E}[V_1^{(2,1)}]\,\textbf{E}\bigg[\sum_{i\ge1}\rho_i\bigg]\\
\noalign{\vglue-1.5mm}
&\qquad\quad{}+2\textbf{E}[\varepsilon(1-\varepsilon)]\,\textbf{E}[V_2^{(2,1)}]\,\textbf{E}\bigg[\sum_{i\ge1}\rho_i\bigg]+\textbf{E}[(1-\varepsilon)^2]\,\textbf{E}[P_{(1,1)}(\rho_1,\rho_2,\ldots)]\\
&\quad{}=2\,\frac{3\cdot4}{(3+\theta)(4+\theta)}\,\frac16+2\,\frac{3\,\theta}{(3+\theta)(4+\theta)}\,\frac23+2\,\frac{3\,\theta}{(3+\theta)(4+\theta)}\,\frac13\\
&\qquad\quad{}+\frac{\theta(1+\theta)}{(3+\theta)(4+\theta)}\,\frac{\theta}{1+\theta}\\
&\quad{}=1-\frac{8+\theta}{(3+\theta)(4+\theta)},
\end{align*}
where $\varepsilon$ is beta$(3,\theta)$, $(V_1^{(2,1)},V_2^{(2,1)})$ is Dirichlet$(2,1)$, $(\rho_1,\rho_2.\ldots)$ is $\text{PD}_\theta$, and all are independent.  This matches \eqref{1,1,2,1-example}.

We can now treat the general case.

\begin{lemma}[\citeauthor{F10}, \citeyear{F10}, Lemma 5.9]\label{three}
For $\bm m=(m_1,\ldots,m_k)$ with $m_1\ge\cdots\ge m_k\ge1$, $\bm n=(n_1,\ldots,n_l)$ with $n_1\ge\cdots\ge n_l\ge1$, $k,l\in{\bf N}$, $m=|\bm m|$, and $n=|\bm n|$,
\begin{align}\label{RHS-2}
&\textbf{E}[P_{\bm m}(\rho(\varepsilon V_1^{\bm n},\ldots,\varepsilon V_l^{\bm n},(1-\varepsilon)\rho_1,(1-\varepsilon)\rho_2,\ldots))]\nonumber\\
&\quad{}=C(\bm m)\sum_{r=0}^{k\wedge l}\bigg[\sum_{\substack{I\subset\{1,\ldots,k\}:\\ |I|=r}}\;\sum_{\substack{\gamma:I\mapsto\{1,\ldots,l\}\\ \text{\rm one-to-one}}}\nonumber\\
&\qquad\qquad\qquad\quad\dfrac{\prod_{i\in I}(m_i+n_{\gamma(i)}-1)!\,\prod_{i\in I^c}(m_i-1)!}{\prod_{i\in I}(n_{\gamma(i)}-1)!}\bigg]\,\frac{\theta^{k-r}}{(n+\theta)_{(m)}}.
\end{align}
\end{lemma}

\begin{proof}
We calculate
\begin{align}\label{prelim}
&\textbf{E}[P_{\bm m}(\rho(\varepsilon V_1^{\bm n},\ldots,\varepsilon V_l^{\bm n},(1-\varepsilon)\rho_1,(1-\varepsilon)\rho_2,\ldots))]\nonumber\\
&\quad{}=C(\bm m)\int_{\overline{\nabla}_\infty}P_{\bm m}^0(\bm p)\,\textbf{P}(\rho(\varepsilon V_1^{\bm n},\ldots,\varepsilon V_l^{\bm n},(1-\varepsilon)\rho_1,(1-\varepsilon)\rho_2,\ldots)\in d\bm p)\nonumber\\
&\quad{}=C(\bm m)\int_{\overline{\nabla}_\infty}\;\sum_{r=0}^{k\wedge l}\;\sum_{\substack{I\subset\{1,\ldots,k\}:\\ |I|=r}}\;\sum_{\substack{\gamma:I\mapsto\{1,\ldots,l\}\\ \text{one-to-one}}}\;\sum_{\substack{\lambda:I^c\mapsto\{l+1,l+2,\ldots\}\\ \text{one-to-one}}}\;\prod_{i\in I}p_{\gamma(i)}^{m_i}\prod_{i\in I^c}p_{\lambda(i)}^{m_i}\nonumber\\
&\qquad\qquad\qquad\qquad\qquad{}\textbf{P}(\rho(\varepsilon V_1^{\bm n},\ldots,\varepsilon V_l^{\bm n},(1-\varepsilon)\rho_1,(1-\varepsilon)\rho_2,\ldots)\in d\bm p)\nonumber\\
&\quad{}=C(\bm m)\sum_{r=0}^{k\wedge l}\;\sum_{\substack{I\subset\{1,\ldots,k\}:\\ |I|=r}}\;\sum_{\substack{\gamma:I\mapsto\{1,\ldots,l\}\\ \text{one-to-one}}}\textbf{E}[\varepsilon^{\sum_{i\in I}m_i}(1-\varepsilon)^{\sum_{i\in I^c}m_i}]\nonumber\\
&\qquad\qquad\qquad\qquad\qquad\qquad\quad\;{}\cdot\textbf{E}\bigg[\prod_{i\in I}(V_{\gamma(i)}^{\bm n})^{m_i}\bigg]\,\textbf{E}\bigg[\sum_{\substack{\lambda:I^c\mapsto\textbf{N}\\ \text{one-to-one}}}\prod_{i\in I^c}\rho_{\lambda(i)}^{m_i}\bigg].
\end{align}
The three remaining expectations are readily evaluated: First,
\begin{align*}
\textbf{E}[\varepsilon^{\sum_{i\in I}m_i}(1-\varepsilon)^{\sum_{i\in I^c}m_i}]&=\frac{\Gamma(n+\theta)}{\Gamma(n)\Gamma(\theta)}\,\frac{\Gamma(n+m_0)\Gamma(\theta+m-m_0)}{\Gamma(n+m+\theta)}\\
&=\frac{n_{(m_0)}\,\theta_{(m-m_0)}}{(n+\theta)_{(m)}},
\end{align*}
where $m_0:=\sum_{i\in I}m_i$.  Second,
\begin{align*}
\textbf{E}\bigg[\prod_{i\in I}(V_{\gamma(i)}^{\bm n})^{m_i}\bigg]&=\frac{\Gamma(n)}{\Gamma(n_1)\cdots\Gamma(n_l)}\,\dfrac{\prod_{i\in I}\Gamma(m_i+n_{\gamma(i)})\prod_{i\in\text{Range}(\gamma)^c}\Gamma(n_i)}{\Gamma(m_0+n)}\\
&=\frac{1}{n_{(m_0)}}\,\dfrac{\prod_{i\in I}(m_i+n_{\gamma(i)}-1)!}{\prod_{i\in I}(n_{\gamma(i)}-1)!},
\end{align*}
and third, with $\widetilde{\bm m}\in\textbf{N}^{k-r}$ defined to have components $m_i$ for $i\in I^c$,
\begin{equation*}
\textbf{E}\bigg[\sum_{\substack{\lambda:I^c\mapsto\textbf{N}\\ \text{one-to-one}}}\prod_{i\in I^c}\rho_{\lambda(i)}^{m_i}\bigg]=\int_{\overline{\nabla}_\infty}P_{\widetilde{\bm m}}^0\,d\text{PD}_\theta=\prod_{i\in I^c}(m_i-1)!\,\frac{\theta^{k-r}}{\theta_{(m-m_0)}}.
\end{equation*}
Substituting these expectations into \eqref{prelim}, we obtain \eqref{RHS-2}.
\end{proof}

Finally, Lemmas \ref{two} and \ref{three} yield \eqref{m-n-equation}, completing the proof of Theorem~\ref{one}. \qed

%\newpage

\section{Applicability of the Rogers--Pitman theorem}\label{sec:proof-RP}

To use Theorem~\ref{RP-Th} in the present setting, we would need to define the one-step transition function $\lambda$ from $\overline{\nabla}_\infty$ to $\mathscr{P}(S)$, and a natural candidate is
\begin{equation*}
\lambda(\bm p,\cdot):=\textbf{P}\bigg(\sum_{i=1}^\infty p_i\,\delta_{\zeta_i}+p_0\,\nu_0\in\cdot\bigg),
\end{equation*}
where $\zeta_1,\zeta_2,\ldots$ are i.i.d.\ $\nu_0$ (recall \eqref{p_0}).  Indeed, $\lambda(\bm p,\cdot)$ is the distribution of a random Borel probability measure $\mu$ on $S$ with $\Phi(\mu)=\bm p$ almost surely.

Since condition \eqref{Dynkin-cond} of Theorem~\ref{Dynkin-Th} holds, \eqref{Q-def-RP} reduces to \eqref{Q-def}, which has already been established by proving Theorem~\ref{one}.  Thus, it would remain only to confirm \eqref{RP-cond-2}, that is, to show that, for each $t>0$ and $\bm p\in\overline{\nabla}_\infty$,
\begin{equation*}
\int_{\mathscr{P}(S)} P(t,\mu,\cdot)\,\lambda(\bm p,d\mu)=\int_{\overline{\nabla}_\infty}\lambda(\bm q,\cdot)\,P(t,\nu,\Phi^{-1}(d\bm q)).
\end{equation*}
where $\nu\in\Phi^{-1}(\bm p)$.  By \eqref{P} and \eqref{Gamma}, this holds if, for each $n\ge0$ and $\bm p\in\overline{\nabla}_\infty$,
\begin{equation}\label{RP-cond_n}
\int_{\mathscr{P}(S)} \Gamma_n^{\theta,\nu_0}(\mu,\cdot)\,\lambda(\bm p,d\mu)=\int_{\overline{\nabla}_\infty}\lambda(\bm q,\cdot)\,\Gamma_n^{\theta,\nu_0}(\nu,\Phi^{-1}(d\bm q)),
\end{equation}
where $\nu\in\Phi^{-1}(\bm p)$.

For $n=0$, this is clear:  The left side of \eqref{RP-cond_n} is $\Pi_{\theta,\nu_0}(\cdot)$, and the right side is
\begin{align*}
&\int_{\overline{\nabla}_\infty}\textbf{P}\bigg(\sum_{i=1}^\infty q_i\,\delta_{\zeta_i}+q_0\,\nu_0\in\cdot\bigg)\textbf{P}\bigg(\sum_{i=1}^\infty \rho_i\,\delta_{\xi_i}\in\Phi^{-1}(d\bm q)\bigg)\\
&\quad{}=\int_{\overline{\nabla}_\infty}\textbf{P}\bigg(\sum_{i=1}^\infty q_i\,\delta_{\zeta_i}+q_0\,\nu_0\in\cdot\bigg)\textbf{P}((\rho_1,\rho_2,\ldots)\in d\bm q)\\
&\quad{}=\textbf{P}\bigg(\sum_{i=1}^\infty \rho_i\,\delta_{\zeta_i}\in\cdot\bigg)\\
&\quad{}=\Pi_{\theta,\nu_0}(\cdot),
\end{align*}
where $(\rho_1,\rho_2,\ldots)$ is $\text{PD}_\theta$, $\xi_1,\xi_2,\ldots$ are i.i.d.\ $\nu_0$, $\zeta_1,\zeta_2,\ldots$ are i.i.d.\ $\nu_0$, and all are independent.
For $n\ge1$, we confirm \eqref{RP-cond_n} using
\begin{align*}
\Gamma_n^{\theta,\nu_0}(\mu,\cdot)&:=\int_{S^n}\mu^n(dx_1\times\cdots\times dx_n)\,\Pi_{n+\theta,(n+\theta)^{-1}(\delta_{x_1}+\cdots+\delta_{x_n}+\theta\nu_0)}(\cdot)\\
&\;=\int_{S^n}\mu^n(dx_1\times\cdots\times dx_n)\,\textbf{P}\bigg(\varepsilon\sum_{i=1}^n U_i^{(n)}\delta_{x_i}+(1-\varepsilon)\sum_{i=1}^\infty\rho_i\,\delta_{\xi_i}\in\cdot\bigg),
\end{align*}
where $\varepsilon$ is beta$(n,\theta)$, $(U_1^{(n)},\ldots,U_n^{(n)})$ is uniform over $\Delta_n$, $(\rho_1,\rho_2,\ldots)$ is $\text{PD}_\theta$, $\xi_1,\xi_2,\ldots$ are i.i.d.\ $\nu_0$, and all are independent.

We begin with the left side of \eqref{RP-cond_n}.  We get
\begin{align}\label{LHS_n}
&\int_{\mathscr{P}(S)}\int_{S^n}\mu^n(dx_1\times\cdots\times dx_n)\,\textbf{P}\bigg(\varepsilon\sum_{i=1}^n U_i^{(n)}\delta_{x_i}+(1-\varepsilon)\sum_{i=1}^\infty\rho_i\,\delta_{\xi_i}\in\cdot\bigg)\nonumber\\
&\qquad\qquad{}\cdot\textbf{P}\bigg(\sum_{i=1}^\infty p_i\,\delta_{\zeta_i}+p_0\,\nu_0\in d\mu\bigg)\nonumber\\
&\quad{}=\textbf{E}\bigg[\int_{S^n}\bigg(\sum_{i=1}^\infty p_i\,\delta_{\zeta_i}+p_0\,\nu_0\bigg)^n(dx_1\times\cdots\times dx_n)\nonumber\\
&\qquad\qquad\quad{}\cdot\textbf{P}\bigg(\varepsilon\sum_{i=1}^n U_i^{(n)}\delta_{x_i}+(1-\varepsilon)\sum_{i=1}^\infty\rho_i\,\delta_{\xi_i}\in\cdot\bigg)\bigg]\nonumber\\
&\quad{}=\sum_{\bm n:|\bm n|=n}\textbf{P}\bigg(\varepsilon(V_1^{\bm n}\delta_{\zeta_1}+\cdots+V_l^{\bm n}\delta_{\zeta_l})+(1-\varepsilon)\sum_{i=1}^\infty\rho_i\,\delta_{\xi_i}\in\cdot\bigg)\,P_{\bm n}(\bm p),
\end{align}
where $(V_1^{\bm n},\ldots,V_l^{\bm n})$ is Dirichlet$(\bm n)$, and the others are independent and distributed as before.  The last equality in \eqref{LHS_n} can be made more transparent by examining the case $n=2$ in more detail.  Let
\begin{equation*}
h(x_1,x_2):=\textbf{P}\bigg(\varepsilon\sum_{i=1}^2 U_i^{(2)}\delta_{x_i}+(1-\varepsilon)\sum_{i=1}^\infty\rho_i\,\delta_{\xi_i}\in\cdot\bigg).
\end{equation*}
Then
\begin{align*}
&\textbf{E}\bigg[\int_{S^2}\bigg(\sum_{i=1}^\infty p_i\,\delta_{\zeta_i}+p_0\,\nu_0\bigg)^2(dx_1\times dx_2)\,h(x_1,x_2)\bigg]\nonumber\\
&\quad{}=\sum_{i=1}^\infty p_i^2\,\text{E}\bigg[\int_{S^2}(\delta_{\zeta_i}\times\delta_{\zeta_i})(dx_1\times dx_2)\,h(x_1,x_2)\bigg]\\
&\qquad\quad{}+\sum_{i,j\ge1:i\ne j}p_ip_j\,\text{E}\bigg[\int_{S^2}(\delta_{\zeta_i}\times\delta_{\zeta_j})(dx_1\times dx_2)\,h(x_1,x_2)\bigg]\\
&\qquad\quad{}+\sum_{i=1}^\infty p_ip_0\,\text{E}\bigg[\int_{S^2}(\delta_{\zeta_i}\times\nu_0)(dx_1\times dx_2)\,h(x_1,x_2)\bigg]\\
&\qquad\quad{}+\sum_{j=1}^\infty p_0p_j\,\text{E}\bigg[\int_{S^2}(\nu_0\times\delta_{\zeta_j})(dx_1\times dx_2)\,h(x_1,x_2)\bigg]\\
&\qquad\quad{}+p_0^2\,\text{E}\bigg[\int_{S^2}(\nu_0\times\nu_0)(dx_1\times dx_2)\,h(x_1,x_2)\bigg]\\
&\quad{}=\sum_{i=1}^\infty p_i^2\,\textbf{P}\bigg(\varepsilon\delta_{\zeta_i}+(1-\varepsilon)\sum_{k=1}^\infty\rho_k\,\delta_{\xi_k}\in\cdot\bigg)\\
&\qquad\quad{}+\sum_{i,j\ge1:i\ne j}p_ip_j\textbf{P}\bigg(\varepsilon[U\delta_{\zeta_i}+(1-U)\delta_{\zeta_j}]+(1-\varepsilon)\sum_{k=1}^\infty\rho_k\,\delta_{\xi_k}\in\cdot\bigg)\\
&\qquad\quad{}+\sum_{i=1}^\infty p_ip_0\,\textbf{P}\bigg(\varepsilon[U\delta_{\zeta_i}+(1-U)\delta_{\zeta'}]+(1-\varepsilon)\sum_{k=1}^\infty\rho_k\,\delta_{\xi_k}\in\cdot\bigg)\\
&\qquad\quad{}+\sum_{j=1}^\infty p_0p_j\,\textbf{P}\bigg(\varepsilon[U\delta_{\zeta}+(1-U)\delta_{\zeta_j}]+(1-\varepsilon)\sum_{k=1}^\infty\rho_k\,\delta_{\xi_k}\in\cdot\bigg)\\
&\qquad\quad{}+p_0^2\,\textbf{P}\bigg(\varepsilon[U\delta_{\zeta}+(1-U)\delta_{\zeta'}]+(1-\varepsilon)\sum_{k=1}^\infty\rho_k\,\delta_{\xi_k}\in\cdot\bigg)\\
&\quad{}=\textbf{P}\bigg(\varepsilon\delta_{\zeta_1}+(1-\varepsilon)\sum_{i=1}^\infty\rho_i\,\delta_{\xi_i}\in\cdot\bigg)\,P_{(2)}(\bm p)\\
&\qquad\quad{}+\textbf{P}\bigg(\varepsilon[U\delta_{\zeta_1}+(1-U)\delta_{\zeta_2}]+(1-\varepsilon)\sum_{i=1}^\infty\rho_i\,\delta_{\xi_i}\in\cdot\bigg)\,P_{(1,1)}(\bm p),
\end{align*}
where $\varepsilon$ is beta$(2,\theta)$, $\zeta_1,\zeta_2,\ldots$ as well as $\zeta$, $\zeta'$, and $\xi_1,\xi_2,\ldots$ are i.i.d.\ $\nu_0$, $U$ is uniform$(0,1)$, $(\rho_1,\rho_2,\ldots)$ is $\text{PD}_\theta$, and all are independent.

And for the right side of \eqref{RP-cond_n}, we get, with $\nu\in\Phi^{-1}(\bm p)$,
\begin{align}\label{RHS_n}
&\int_{\overline{\nabla}_\infty}\textbf{P}\bigg(\sum_{i=1}^\infty q_i\,\delta_{\zeta_i}+q_0\,\nu_0\in\cdot\bigg)\nonumber\\
&\qquad{}\cdot\int_{S^n}\nu^n(dx_1\times\cdots\times dx_n)\,\textbf{P}\bigg(\varepsilon\sum_{i=1}^n U_i^{(n)}\delta_{x_i}+(1-\varepsilon)\sum_{i=1}^\infty\rho_i\,\delta_{\xi_i}\in\Phi^{-1}(d\bm q)\bigg)\nonumber\\
&\quad{}=\sum_{\bm n:|\bm n|=n}\int_{\overline{\nabla}_\infty}\textbf{P}\bigg(\sum_{i=1}^\infty q_i\,\delta_{\zeta_i}+q_0\,\nu_0\in\cdot\bigg)\nonumber\\
&\qquad\qquad\qquad\quad\cdot\textbf{P}(\rho(\varepsilon V_1^{\bm n},\ldots,\varepsilon V_l^{\bm n},(1-\varepsilon)\rho_1,(1-\varepsilon)\rho_2,\ldots)\in d\bm q)\,P_{\bm n}(\bm p)\nonumber\\
&\quad{}=\sum_{\bm n:|\bm n|=n}\textbf{P}\bigg(\varepsilon(V_1^{\bm n}\delta_{\zeta_1}+\cdots+V_l^{\bm n}\delta_{\zeta_l})+(1-\varepsilon)\sum_{i=1}^\infty\rho_i\,\delta_{\xi_i}\in\cdot\bigg)\,P_{\bm n}(\bm p).
\end{align}
So \eqref{LHS_n} and \eqref{RHS_n} are equal, proving \eqref{RP-cond_n} and confirming that \eqref{RP-cond-2} is satisfied.

\end{document}